\documentclass[12pt]{amsart}
\usepackage{amssymb}

\usepackage{tikz}
\newtheorem{Teo}{Theorem}[section]
\newtheorem*{theorem*}{Theorem}
\newtheorem{Lem}[Teo]{Lemma}
\newtheorem{Obs}[Teo]{Remark}
\newtheorem{Def}[Teo]{Definition}

\newtheorem{Cor}[Teo]{Corollary}

\newtheorem{Con}[Teo]{Conjecture}
\newtheorem{lem-def}[Teo]{Lemma-Definition}

\DeclareRobustCommand\longtwoheadrightarrow
{\relbar\joinrel\twoheadrightarrow}

\renewenvironment{proof}{{\bfseries Proof.}}{\qed}
\newcommand{\N}{\mathbb N}
\newcommand{\Z}{\mathbb Z}
\newcommand{\Q}{\mathbb Q}

\def\op{\operatorname}

\def\aa{\mathcal{A}}

\def\al{\alpha}

\def\ars#1{\renewcommand\arraystretch{#1}}

\def\bb{{\mathcal B}}
\def\be{\beta}

\def\bs{\vskip.5cm}
\def\btd{B(\t,\dta)}

\def\dd{\mathcal{D}}
\def\defn{\nn{\bf Definition. }}
\def\dep{\op{\mbox{\rm\small depth}}}

\def\dm{\Delta_\mu}

\def\diso{\lower.4ex\hbox{$\downarrow$}\raise.4ex\hbox{\mbox{\scriptsize
$\wr$}}}

\def\dta{\delta}

\def\e{\medskip}

\def\ep#1{\exp(\Pi i#1)}
\def\ep{\epsilon}

\def\g{\Gamma}
\def\ga{\gamma}

\def\gen#1{\big\langle\, {#1} \,\big\rangle}

\def\ggm{\mathcal{G}_\mu}

\def\ggn{\mathcal{G}_\nu}

\def\gi{\g_{\infty}}

\def\gm{\g_\mu}

\def\inv{\op{in}_v}
\def\imp{\ \Longrightarrow\ }
\def\im{\op{Im}}

\def\inm{\op{in}_\mu}
\def\inn{\op{in}}

\def\inu{\op{in}_\nu}

\def\ism{\lower.3ex\hbox{\ars{.08}$\begin{array}{c}\,\to\\\mbox{\tiny $\sim\,$}\end{array}$}}
\def\iso{\ \lower.3ex\hbox{\ars{.08}$\begin{array}{c}\lra\\\mbox{\tiny $\sim\,$}\end{array}$}\ }

\def\ka{\kappa}

\def\kb{\overline{K}}
\def\kbx{\overline{K}[x]}

\def\khx{K^h[x]}
\def\km{k_\mu}

\def\kp{\op{KP}}

\def\kpm{\op{KP}(\mu)}

\def\kx{K[x]}

\def\lg{l\raise.6ex\hbox to.2em{\hss.\hss}l}

\def\lra{\,\longrightarrow\,}

\def\md#1{\; \mbox{\rm(mod }{#1})}

\def\mlv{Mac Lane--Vaqui\'e\ }

\def\mub{\bar{\mu}}

\def\nn{\noindent}

\def\oo{\mathcal{O}}
\def\orb{\hbox to  .3em{$\backslash$}\backslash}
\def\ord{\op{ord}}
\def\p{\mathfrak{p}}

\def\sg{\sigma}

\def\sii{\ \Longleftrightarrow\ }

\def\sub{\subseteq}
\def\supp{\op{supp}}

\def\t{\theta}

\def\ttt{\mathcal{T}}

\def\vt{v_\theta}
\def\vtd{v_{\theta,\delta}}

\newcounter{cs}
\stepcounter{cs}
\newcommand{\casos}{\begin{itemize}}
\newcommand{\fcasos}{\end{itemize}\setcounter{cs}{1}}

\newfont{\tit}{cmr12 scaled \magstep3}

\title{Depth of extensions of valuations}

\makeatletter
\@namedef{subjclassname@2010}{%
  \textup{2010} Mathematics Subject Classification}

\author[Nart]{Enric Nart}
\address{Departament de Matem\`{a}tiques,         Universitat Aut\`{o}noma de Barcelona,         Edifici C, E-08193 Bellaterra, Barcelona, Catalonia}
\email{enric.nart@uab.cat}

\author[Novacoski]{Josnei Novacoski}
\address{Departamento de Matem\'{a}tica,         Universidade Federal de S\~ao Carlos, Rod. Washington Luís, 235, 13565--905, S\~ao Carlos -SP, Brazil}
\email{josnei@ufscar.br}

\thanks{Partially supported by grant PID2020-116542GB-I00  funded by the Spanish MCIN/AEI. During the realization of this project the second author was supported by a grant from Funda\c{c}\~ao de Amparo \`a Pesquisa do Estado de S\~ao Paulo (process number 2024/08989-6).}


\begin{document}

\begin{abstract}
In this paper we develop the theory of the depth of a simple algebraic extension of valued fields $(L/K,v)$. This is defined as the minimal number of augmentations appearing in some Mac Lane-Vaquié chain for the valuation on $\kx$ determined by the choice of some generator of the extension. In the defectless and unibranched case, this concept leads to a generalization of a classical result of Ore about the existence of $p$-regular generators for number fields. Also, we find what  valuation-theoretic conditions characterize the extensions having depth one.
\end{abstract}

\maketitle



\section*{Introduction}
Let $K$ be a field, $\kb$ an algebraic closure of $K$ and $v$ a valuation on $\kb$. For any subfield $K\sub L\sub \kb$, we denote the \textit{value group} and \textit{residue field} of $(L,v)$ by $vL$ and $Lv$, respectively.
Let  $K^h\subseteq \overline K$ be the Henselization of $(K,v)$.

For some $\t\in\kb$, let  $g\in\kx$ be its minimal 
polynomial over $K$ and consider the extension $L=K(\t)$ of $K$. The valued field $(L,v)$ singles out a monic irreducible factor $G\in\khx$ of $g$. The germ of this idea goes back to Hensel.   

The arithmetic data associated to $(L/K,v)$ are linked by:
\[
\deg(G)=e(L/K)\,f(L/K)\,d(L/K),
\]
where $e(L/K)$ is the \textit{ramification index}, $f(L/K)$ is the \textit{inertia degree} and $d(L/K)$ is the 
{defect} of $(L/K,v)$. 

Inspired in the work of Ore, Mac Lane realized that the arithmetic invariants of $(L,v)$  could be described in terms of the following valuation on $\kx$:
\[
\vt\colon \kx\lra \g\cup\{\infty\},\qquad f\longmapsto \vt(f)=v(f(\t)),
\]  
where $\g=v\kb$. Since $\vt^{-1}(\infty)=g\kx$, the valuations  $\vt$ and $v_{\mid L}$ are determined one by each other through
\[
\vt\colon \kx \longtwoheadrightarrow \kx/(g)\stackrel{\sim}\lra L\stackrel{v}\lra \g\cup\{\infty\}, 
\]
where the isomorphism $\kx/(g)\stackrel{\sim}\to L $ is induced by $x\mapsto \t$.

By a celebrated theorem of Mac Lane, and its generalization by Vaqui\'e, the valuation $\vt$ can be constructed from $v$ by applying a finite number of \textit{augmentations} of valuations on $\kx$ whose restriction to $K$ is $v$:
\[
v\ \stackrel{\phi_0,\ga_0}\lra\ 	\mu_0\ \stackrel{\phi_1,\ga_1}\lra\  \mu_1\ \lra\ \cdots
	\ \stackrel{\phi_{r-1},\ga_{r-1}}\lra\ \mu_{r-1}
	\ \stackrel{\phi_{r},\ga_{r}}\lra\ \mu_{r}=\vt,
\]
where $\phi_0,\dots,\phi_{r-1}\in\kx$ are certain (abstract) \textit{key polynomials} for $\vt$ and $\phi_r=g$; while $\ga_0,\dots,\ga_{r-1}\in\g$ and $\ga_r=\infty$. The initial step $v\to\mu_0$ is a formal augmentation. The monic polynomial $\phi_0$ has degree one  and $\mu_0$ is  determined by the conditions
\[
(\mu_0)_{\mid K}=v,\qquad \mu_0(\phi_0)=\ga_0,\qquad \mu_0\left(\sum\nolimits_{0\le n}a_n\phi_0^n\right)=\min_{0\le n}\{\mu_0\left(a_n\phi_0^n\right)\}.
\]
Moreover, each augmentation $\mu_{n-1} 
\to \mu_n$ is of one of the following two types:\e

\nn\textbf{Ordinary}: \ $\phi_{n}$ is a (Mac Lane-Vaqui\'e) \textit{key polynomial} for $\mu_{n-1}$.\e

\nn\textbf{Limit}:  \ $\phi_{n}$  is a (Mac Lane-Vaqui\'e) \textit{limit key polynomial} for a certain increasing family  $\aa_n$ of valuations, admitting $\mu_{n-1}$ as its first valuation.\e

In both cases, we have $\mu_n(\phi_n)=\ga_n$.
The chain is said to be  a \textit{\mlv chain} (abbreviated MLV chain)  if it additionally satisfies:
\begin{itemize}
\item $\,1=\deg \phi_0<\cdots<\deg \phi_r$.
\item \ If $\,\mu_{n-1}\to\mu_{n}\,$ is a limit augmentation, then $\vt(\phi_{n-1})=\ga_{n-1}$ and the valuations in $\aa_n$ have constant degree equal to $\deg \phi_{n-1}$.
\end{itemize}

All MLV chains of $\vt$ have the same length $r$ (\cite[Theorem 4.3]{MLV}). This common length is said to be the \textbf{depth} of $\vt$ (or $\t$) and we write  
\[
\dep(\t):=\dep(\vt):=r.
\]

As shown in \cite[Sec.6]{NN}, we have
\begin{itemize}
	\item $e(L/K)=e(\mu_0)e(\mu_1)\cdots e(\mu_{r-1})$.
	\item $f(L/K)=f(\mu_0\to\mu_1)\cdots f(\mu_{r-1}\to\mu_r)$.
	\item $d(L/K)=d(\mu_0\to\mu_1)\cdots d(\mu_{r-1}\to\mu_r)$.
\end{itemize}
For the definition of these ``relative" numerical data, we address the reader to \cite{NN}.

Therefore, the study of the numerical data of $(L/K,v)$ can be splitted into the analysis of small pieces given by single augmentations $\mu\to\nu$ of valuations on $\kx$. 

 The concept of depth is not intrinsically associated to $(L/K,v)$. Different  generators of the same extension may have different depths. This leads to defining the \textbf{depth of an extension of valuations} as
\[
\dep(L/K,v):=\min\{\dep(\t)\mid L=K(\t)\}.
\]

The aim of this paper is to study $\dep(L/K,v)$. Our main motivation for this study comes from number theory. In a pioneering work, Ore conjectured the existence of an  algorithm to compute prime ideal decomposition in number fields by the concatenation of certain procedures, classically known as \emph{dissections}  \cite{ore1,ore2}. These dissections were of two kinds, either determined by the sides of some Newton polygons with respect to certain valuations on $\Q[x]$, or by the factorization over finite fields  of some polynomials associated to the different sides of these polygons.  

For a fixed prime $p$, he showed that every number field admits a \emph{$p$-regular} generator over $\Q$, meaning that the prime ideal decomposition of $p$ can be achieved after at most three dissections. Also, he conjectured that if one has to  work with an arbitrary generator, then the prime ideal decomposition of $p$ should be always achieved after a finite number of dissections. 

Mac Lane proved Ore's conjecture  after extending his ideas to the more general setting of the computation of the extensions of a discrete rank-one valuation $v$ on an arbitrary field $K$, to separable finite extensions $L/K$ \cite{mcla,mclb}.  In  this approach, the classical dissections were replaced with suitable augmentations of valuations on $\kx$.

Let $\Q_p$ be the $p$-adic field. In our setting, Ore's result about the existence of $p$-regular generators can  be reinterpreted as follows.  

\begin{theorem*}
For every finite extension $L/\Q_p$ we have $\dep(L/\Q_p,\ord_p)\le 2$.	
\end{theorem*}
	
One of the main results of this paper (Theorem \ref{NewOre}) is a generalization this result to defectless extensions of a Henselian valued field $(K,v)$. Also, we characterize unibranched defectless extensions of depth one (Theorem \ref{dep1}).




The main tool used for computing the depth are Okutsu sequences. Okutsu sequences of algebraic elements were introduced by Okutsu for complete, discrete, rank-one valued fields \cite{Ok}. Under the form of \textit{distinguished  chains}, these objects have been developed by several authors, mainly in the Henselian case \cite{AK1, AK2, PZ}. The equivalence between Okutsu sequences and distinguished chains was established  in \cite{OkS} for defectless extensions of Henselian valued fields of an arbitrary rank. 

For arbitrary extensions of Henselian valued fields, Okutsu sequences were considered in \cite{OS}. In this paper, we introduce Okutsu sequences in full generality and we use the results of \cite{NNP} to show that the length of any Okutsu sequence is equal to the depth of the last element in the sequence (Theorem \ref{OSdepth}). This result provides a very handy method to compute the depth of generators of $L/K$ and, consequently, the depth of $(L/K,v)$.


\section{Okutsu sequences and depth}\label{secOS}
Recall that $\g=v\kb$. From now on, we denote $\g\cup\{\infty\}$ simply by $\gi$.
 
\subsection{Distances between algebraic elements}
 Let us fix an algebraic element $\t\in\kb$ 
 of degree $n\ge 1$.
For every integer $1\le m\le n$, we define the \textbf{set of distances} of $\t$ to elements in $\kb$ of degree $m$ over $K$ as:
\[
D_m=D_m(\t,K):=\left\{v(\t-b)\mid b\in\kb,\ \deg_Kb=m\right\}\sub\gi.
\] 
Note that $\max(D_n)=\infty$.

The set $D_1$ has been extensively studied by Blaszczok and Kuhlmann for its connexions with defect and immediate extensions \cite{B,Kuhl}.

We are interested in $\max\left(D_m\right)$, the maximal distance of $\t$ to elements of a fixed degree over $K$. Since this maximal distance may not exist, we follow \cite{Kuhl} and we replace this concept with an analogous one in the context of \textit{cuts} of $\g$.

For $S,T\sub \g$ and $\ga\in \g$, the following expressions
$$\ga<S,\ \quad \ga>S,\ \quad \ga\le S,\ \quad \ga\ge S,\ \quad S<T$$
mean that the corresponding inequality holds for all $\al\in S$ and all $\be\in T$.



A \textbf{cut} in $\g$ is a pair $\dta=(\dta^L,\dta^R)$ of subsets of $\g$ such that $$\dta^L< \dta^R\quad\mbox{ and }\quad \dta^L\cup \dta^R=\g.$$ 
Note that $\dta^R=\g\setminus\dta^L$. For all $D\sub \g$ we denote by  $D^+$, $D^-$ the cuts determined by
\[
\dta^L=\{\ga\in \g\mid \exists \al\in D: \ga\leq \al\},\qquad \dta^L=\{\ga\in \g\mid \ga<D\},
\]
respectively.
If $D=\{\ga\}$, then we will write $\ga^+=(\g_{\le \ga},\g_{>\ga})$ instead of $\{\ga\}^+$ and $\ga^-=(\g_{<\ga}, \g_{\ge \ga})$ instead of $\{\ga\}^-$. These cuts are said to be 
\textbf{principal}. 



The set of cuts is totally ordered with respect to the following ordering:
\[
(\dta^L,\dta^R)\le (\ep^L,\ep^R)\ \sii\ \dta^L\sub \ep^L.
\]
The \textbf{improper} cuts 
\[
-\infty:=(\emptyset,\g),\qquad \infty^-:=(\g,\emptyset),
\] 
are the absolute minimal and absolute maximal cuts, respectively.

Also, there is an ``addition" of cuts, defined as follows:
\[
(\dta^L,\dta^R)+(\ep^L,\ep^R)=\left(\dta^L+\ep^L,\g\setminus(\dta^L+\ep^L) \right).
\] 
Note that $\dta+\infty^-=\infty^-$ for every cut $\dta$. For every $\ga\in\g$, the cut $\ga^++\dta=\left(\ga+\dta^L,\ga+\dta^R\right)$ will be denoted simply by $\ga+\dta$.

\begin{Def}\label{distance}
For $m<n$, we define $d_m(\t)$ to be the cut $D_m^+$ of $\,\g$.  

We agree that $d_n(\t)$ is the improper cut $\infty^-$.
\end{Def}

For instance, if $D_m$ has a maximal element $\ga\in\g$, then  $d_m(\t)=\ga^+$. 

Finally, let us recall some basic properties of  the function $d_1$, which follow easily fom the definitions.

\begin{Lem}\label{d1+}
Let $\t,\eta\in\kb$.
\begin{enumerate}
	\item[(a)] $d_1(\t+\eta)\ge \min\{d_1(\t),d_1(\eta)\}
	$ and equality holds if $d_1(\t)\ne d_1(\eta)$.
	\item [(b)] If $a\in K$, then $d_1(a\t)=v(a)+d_1(\t)$.
\end{enumerate}
\end{Lem}

\subsection{Okutsu sequences}
As mentioned in the Introduction, the valuation $\vt$ on $\kx$ contains relevant information about the extension of $v$ to the field $L=K(\t)$. This information can be captured as well by certain sequences of sets of algebraic elements. 

We say that a  subset $A\sub\kb$ has a \textbf{common degree} if all its elements have the same degree over $K$. In this case, we shall denote this common degree by $\deg_K A$. 

\begin{Def}\label{defOkS}
	An \textbf{Okutsu sequence} of $\t$ is a finite sequence 
\[
\left[A_0,A_1,\dots,A_{r-1},A_r=\{\t\}\right], 
\]
of common degree subsets of $\kb$ whose degrees grow strictly:
\begin{equation}\label{degOS}	
		1=m_0<m_1<\cdots<m_r=n,\qquad m_\ell=\deg_K A_\ell, \ \ 0\le\ell\le r,
\end{equation}
	and the following fundamental property is satisfied for all $\,0\le\ell< r$: \e
	
	(OS0) \ For all $b\in\kb$ such that $\deg_Kb<m_{\ell+1}$, we have $v(\t-b)\le  v(\t-a)$ for some $a\in A_\ell$.\e
	
	Moreover, the  following additional properties are imposed, for all $0\le\ell< r$: \e
	
	(OS1) \ $\#A_\ell=1$ whenever $\max\left(D_{m_{\ell}}\right)$ exists.\e
	
	(OS2) \ If $\max\left(D_{m_{\ell}}\right)$ does not exist, then we assume that $A_\ell$ is well-ordered with respect to the following ordering: \ $a<a'\ \sii \ v(\t-a)<v(\t-a')$.\e
	
	(OS3) \ For all $a\in A_{\ell}$, $b\in A_{\ell+1}$, we have $v(\t-a)<v(\t-b)$.
\end{Def}

Let us discuss the existence and construction of Okutsu sequences.
Consider the canonical sequence \textbf{of minimal degrees} of the distances of $\t$:
\begin{equation}\label{degDist}	
1=d_0<d_1<\cdots<d_s=n,
\end{equation}
defined recursively as follows:

$\bullet$ \ $d_0=1$,

$\bullet$ \ for every $\ell\in\N$, $d_\ell$ is the least integer $m>d_{\ell-1}$ such that there exists some $\ep\in D_m$ satisfying $\ep>D_{d_{\ell-1}}$.\e

Now, for each $0\le \ell<  s$ choose subsets $A_\ell\sub\kb$ of common degree $\deg_K A_\ell=d_\ell$ such that
\[
 \{v(\t-a)\mid a\in A_\ell\}\sub D_{d_\ell}
\]
is a well-ordered cofinal subset of $D_{d_\ell}$. Also, if for some $\ell$ there exists $\ga=\max\left(D_{d_\ell}\right)$, then we take $A_\ell=\{a\}$, for some $a\in\kb$ such that $\deg_K a=d_\ell$ and $v(\t-a)=\ga$. 

Finally, for $\ell>0$, we consider in $A_\ell$ only elements $a$ such that $v(\t-a)>D_{d_\ell-1}$. 

Clearly, $\left[A_0,A_1,\dots,A_{s-1},A_s=\{\t\}\right]$ is an Okutsu sequence of $\t$ and, conversely, all Okutsu sequences arise in this way. 

In particular, the sequence (\ref{degOS}) of degrees over $K$ of the sets $A_\ell$ of an Okutsu sequence is equal to the canonical sequence (\ref{degDist}) of minimal degrees of the distances of $\t$. In other words, $r=s$ and $d_\ell=m_\ell$ for all $0\le \ell<r$.

Our aim in this section is to prove the following theorem. 

\begin{Teo}\label{OSdepth}
The length $r$ of any Okutsu sequence of $\t$ is equal to $\dep(\t)$. Moreover, for every MLV chain of $\vt$:
\[
v\ \to\ \mu_0\ \to\ \mu_1\ \to\ \cdots \ \to\ \mu_{r-1}\ \to\ \mu_r=\vt
\]
and every $0\le \ell<r$,
the augmentation $\mu_\ell\ \to\ \mu_{\ell+1}$ is ordinary if  and only if $D_{m_\ell}$ contains a maximal element.
\end{Teo}

This result will follow easily from the results of \cite{NNP}, where $\dep(\t)$ was computed in terms of ultrametric balls.

\subsection{Depth in terms of ultrametric balls}
Consider the set of all closed ultrametric balls centered at $\t$:
\[
\bb:=\left\{\btd\mid \dta\in\gi\right\},\qquad \btd=\{a\in\kb\mid v(\t-a)\ge\dta\}.
\]
The set $\bb$ is totally ordered with respect to descending inclusion and this ordering is coherent with that of $\gi$:
\[
B(\t,\dta)\supseteq B(\t,\ep) \ \sii\ \dta\le\ep. 
\]
The last element of $\bb$ is $B(\t,\infty)=\{\t\}$.

The \textbf{degree} over $K$ of any ball $B\in\bb$ is defined as:
\[
\deg_K B=\min\{\deg_Kb\mid b\in B\}.
\]
 Obviously, the degree grows for more advanced (smaller) balls:
\[
B\supseteq B' \ \imp\ \deg_KB\le\deg_KB'. 
\]

The following lemma shows how to reinterpret the sets of distances of $\t$ in terms of ultrametric balls.

\begin{Lem}\label{B-D}
Let $\dta\in D_m$, for some $1\le m\le n$. Then, $\deg_K \btd\le m$, and equality holds if and only if $m=1$ or $\dta>D_{k}$ for all $k<m$.
\end{Lem}

\begin{proof}
Let $\dta=v(\t-a)$, for some $a\in\kb$ such that $\deg_K a=m$.   Since $a\in\btd$, we have $\deg_K \btd\le m$. If $m=1$, then clearly $\deg_K \btd= m$.

Suppose that $m>1$. The condition $\deg_K \btd=m$ is equivalent to $b\not\in\btd$ for all $b\in\kb$ such that $\deg_Kb<m$. This is clearly equivalent to $\dta>D_{k}$ for all $k<m$.
\end{proof}\e

Clearly, the set of degrees of the balls in $\bb$ is finite and upper-bounded  by $n=\deg_K\t=\deg_KB(\t,\infty)$. Let these degrees be:
\begin{equation}\label{degBalls}
1=n_0<n_1<\cdots<n_{r-1}<n_r=n. 
\end{equation}
For every $m\in\N$, let $\bb_m\sub\bb$ be the subset of all balls of degree $m$. 

To each radius $\dta\in \gi$  we can associate the following valuation on $\kbx$:
\[
\vtd\left(a_0+a_1(x-\t)+\ldots+a_m(x-\t)^m\right):=\min_{0\leq j\leq m}\{v(a_j)+j\delta\}. 
\]
Note that $v_{\t,\infty}=\vt$.
Now, for every $0\le \ell< r$ consider the valuation $\mu_\ell$ on $\kx$ obtained by restriction of $\vtd$, where the radius $\dta\in\g$ is chosen as follows:\e

$\bullet$ \ If $\bb_{n_\ell}$ has a maximal element, then  $\dta$ is the radius of this maximal element. 


$\bullet$ \ If $\bb_{n_\ell}$ has no maximal element, then consider any $B\in \bb_{n_\ell}$ and any $a\in B$ such that $\deg_Ka=n_\ell$. Then, take
\[
\dta=\max\{v(\t-\sg(a))\mid \sg\in \op{Aut}(\kb/K)\}.
\]
Since $B\supseteq \btd$ and the ball $\btd$ contains some conjugate of $a$, we still have  $\deg_K\btd=n_\ell$.

In both cases, the chosen ball $\btd$ of degree $n_\ell$ is $\vt$-\textit{optimal} in the terminology of \cite[Section 4]{NNP}. 

\begin{Teo}\cite[Theorem 7.2]{NNP}\label{ballsDepth}
The following sequence is a MLV chain of $\vt$:
\[
v\ \to\ \mu_0\ \to\ \mu_1\ \to\ \cdots \ \to\ \mu_{r-1}\ \to\ \mu_r=\vt. 
\]
Moreover, for every $0\le\ell<r$, the augmentation $\mu_\ell\ \to\ \mu_{\ell+1}$ is ordinary if and only if  $\bb_{n_\ell}$ contains a last (smallest) ball.
\end{Teo}

Let us show that Theorem \ref{OSdepth} is an immediate consequence of this result. \e

{\bf Proof of Theorem \ref{OSdepth}.}
First of all, the sequence (\ref{degBalls}) of degrees over $K$ of the balls in $\bb$ coincides with the sequence (\ref{degDist}) of minimal distances of $\t$.

Indeed, $d_0=1=n_0$. Suppose that for some $\ell\ge0$ we have $d_\ell=n_\ell$. By definition, $n_{\ell+1}$ is the least integer $m>n_\ell$ such that there exists a ball $B\in\bb$ with $\deg_KB=m$.  By Lemma \ref{B-D}, this is equivalent to saying that $n_{\ell+1}=d_{\ell+1}$ is the least integer $m>d_{\ell}=n_{\ell}$ such that there exists some $\ep\in D_m$ satisfying $\ep>D_{d_{\ell-1}}$. 

Finally, it follows immediately from the definitions that $\bb_{n_\ell}$ has a last ball $\btd$ if and only if $\dta=\max\left(D_{n_\ell}\right)$. \hfill{$\Box$}

\begin{Obs}
The paper \cite{NNP} deals only with balls of finite radius. Hence, Theorem \ref{ballsDepth} is proven in \cite{NNP} only for valuations on $\kx$  with trivial support. However, it is trivial to check that all results of \cite{NNP} remain valid if we admit balls with radius equal to infinity. In particular, our statement of Theorem \ref{ballsDepth} is correct.
\end{Obs}

\section{Residual polynomial operators}

Residual polynomial operators are the key objects  facilitating the application of Theorem \ref{OSdepth} to the computation of the depth of algebraic elements. These operators were introduced by \O. Ore in his pioneering work \cite{ore1, ore2}, and generalized by J. Montes \cite{montes} to a more general valuation-theoretical context. 

All results in this  section could be easily deduced from \cite{KP, MLV}. For the ease of the reader we give a self-contained presentation of this topic, with short proofs.

\subsection{Graded algebra of a valuation on $\kx$}\label{subsecGradedAlgebra}

The graded algebra of the valued field $(K,v)$ is defined as
$\op{gr}_v(K)=\bigoplus_{\dta \in vK}\mathcal P_\dta(v)/\mathcal P_\dta^+(v)$, where
\[
\mathcal P_\dta(v)=\{a\in K\mid v(a)\ge \dta\}\supseteq \mathcal P^+_\dta(v)=\{a\in K\mid v(a)> \dta\}.
\] 
Every $a\in K^*$ has a homogeneous \textbf{initial coefficient}  $\inv a\in\op{gr}_v(K)$, defined as the image of $a$ in $\mathcal P_{v(a)}/\mathcal P_{v(a)}^+$. 

Let $\mu\colon \kx\to \gi$ be an extension of $v$ to $\kx$, taking values in the divisible hull $\g=v\kb$ of $vK$.
The \textbf{support} of $\mu$ is the prime ideal  
\[
\p:=\supp(\mu)=\mu^{-1}(\infty)\in\op{Spec}(\kx). 
\]
The valuation $\mu$ induces in an obvious way a valuation $\mub$ on the field of fractions of $\kx/\p$. We denote by $\gm$ and $\km$ the value group and residue field of $\mu$, defined as the value group and residue field of $\mub$, respectively.  
Since $\gm\sub\g$, the quotient $\gm/vK$ is a torsion group.

The {\bf graded algebra} of $\mu$ is the integral domain $\ggm=\bigoplus_{\ep \in \Gamma_\mu}\mathcal P_\ep(\mu)/\mathcal P_\ep^+(\mu)$, where
\[
\mathcal P_\ep(\mu)=\{f\in K[x]\mid \mu(f)\geq \ep\}\supseteq\mathcal P_\ep^+(\mu)=\{f\in K[x]\mid \mu(f)> \ep\}.
\]
We say that $\ep$ is the \textbf{grade} of the elements in $\mathcal P_\ep(\mu)/\mathcal P_\ep^+(\mu)$.
Every $f\in K[x]\setminus\p$ has a homogeneous initial coefficient  $\inm f\in\ggm$ of grade $\mu(f)$, defined as the image of $f$ in $\mathcal P_{\mu(f)}/\mathcal P_{\mu(f)}^+$. 

If $\p\ne0$, then $\p=g\kx$ for some monic irreducible $g\in \kx$. The choice of a root $\t\in\kb$ of $g$ determines an isomorphism $\kx/\p\simeq K(\t)$, and hence an identification of $\mub$ with an extension of $v$ 
to the field $L:=K(\t)$, via:
\[
\mu\colon \kx \longtwoheadrightarrow \kx/\p\stackrel{\mub}\lra \gi.
\]
We emphasize that $\mu=\vt$ only when the valuation on $L$ induced by $\bar{\mu}$ coincides with $v$.
In any case, there is a canonical isomorphism of $\op{gr}_v(K)$-algebras: 
\begin{equation}\label{natIso}
\ggm\simeq\op{gr}_{\mub}(L),\qquad  \inm f\mapsto \inn_{\mub}\left(f(\t)\right).
\end{equation}
In particular,  $\ggm$ is a \textbf{simple} graded algebra; that is, all nonzero homogeneous elements in $\ggm$ are units.

If $\p=0$ and $\km/K\!v$ is an algebraic extension, then $\mu$ is a valuation of  a very special kind.  The graded algebra $\ggm$ is simple too.   We shall ignore this case because it plays no role in the problems addressed in this paper.

If $\p=0$ and the extension $\km/K\!v$ is transcendental, then we  say that $\mu$ is \textbf{residue-transcendental}. In this case, 
the graded algebra $\ggm$ contains homogeneous prime elements.

\begin{Def}\label{keypolyno}
	A monic $\phi\in K[x]$ is a \textbf{key polynomial}  for $\mu$ if  $(\inm\phi)\ggm$ is a homogeneous prime ideal containing no initial coefficient \ $\inm f$ with  $\deg f< \deg \phi$.
\end{Def}

Equivalently, \cite[Prop. 2.3]{KP} shows that a monic $\phi\in K[x]$ is a key polynomial  for $\mu$ if  $(\inm\phi)\ggm$ is a homogeneous prime ideal and the \emph{truncated function}  $\mu_\phi$ is equal to $\mu$. This means that for all nonzero $f\in\kx$, we have
\begin{equation}\label{muf}
f=\sum\nolimits_{0\le n}a_n \phi^n,\ \ \deg(a_n)<\deg(\phi)\ \imp\ \mu(f)=\min_{0\le n}\left\{\mu\left(a_n \phi^n\right)\right\}.
\end{equation}

We denote by $\kpm$ the set of all key polynomials for $\mu$. These polynomials are  necessarily irreducible in $\kx$.

Let $\ttt(\g)$ be the set of all extensions of $v$ to $\kx$ taking values in $\g$. This set has the structure of  a tree (see \cite{VT}) with respect to the following partial ordering:
\[
\mu\le\nu\ \sii \ \mu(f)\le \nu(f)\quad\mbox{for all }\ f\in\kx. 
\]
Being a tree means that all intervals are totally ordered. 
The maximal elements in $\ttt(\g)$ are said to be \textbf{leaves} of this tree.

The following criterion  for the existence of key polynomials follows from \cite[Proposition 1.4]{Vaq1}.

\begin{Teo}\label{existence}
For every $\mu\in\ttt(\g)$ the following conditions are equivalent.
\begin{enumerate}
\item[(a)] $\kpm\ne\emptyset$.
\item[(b)] $\mu$ is residue-transcendental.
\item [(c)] $\mu$ is not a leaf of $\ttt(\g)$. 
\end{enumerate}
\end{Teo}

\subsection{Residual polynomial operators}\label{subsecOre}

Let $\mu$ be a node of $\ttt(\g)$; that is, $\mu$ is an extension of $v$ to $\kx$ such that $\gm\sub\g$.

Let $\dm=\mathcal{P}_0(\mu)/\mathcal{P}^+_0(\mu)\sub\ggm$ be the subring formed by all homogeneous elements of grade zero in $\ggm$.
Let  $\kappa=\kappa(\mu)$ be the relative algebraic closure of $K\!v$ in $\km$.
There are canonical injective ring homomorphisms 
\[
K\!v\hookrightarrow\kappa\hookrightarrow\dm\hookrightarrow \km.
\] 

The \textbf{relative ramification index} of $\mu$ is defined as 
\[
e=e(\mu):=\left(\gm\colon \gm^0\right),
\]
where $\gm^0$
is the subgroup of all grades of homogeneous units in $\ggm$.

If $\mu$ has nontrivial support, then obviously $\kappa\simeq\dm\simeq\km$. Also, as remarked in the previous section, the graded algebra $\ggm$ is simple; thus,  $\gm^0=\gm$ and $e(\mu)=1$.\e

\emph{From now on, we suppose that $\mu$ is residue-transcendental and we fix a key polynomial $\varphi\in\kpm$ of minimal degree.}\e

 By \cite[Proposition 3.5]{KP}, we have
\[
\gm^0=\left\{\mu(a)\mid a\in \kx,\ 0\le \deg(a)<\deg(\varphi)\right\}.
\]

\noindent{\bf Notation. }For every $a\in \kx$ such that $0\le \deg(a)<\deg(\varphi)$, we shall denote the homogeneous unit $\inm a\in\ggm$ simply  by $\bar{a}$.\e

By (\ref{muf}), $\gm=\gen{\gm^0,\mu(\varphi)}$, so that $e$ is the least positive integer such that $e\mu(\varphi)\in\gm^0$. Let us fix $u\in\kx$, such that $0\le \deg(u)<\deg(\varphi)$ and $\mu(u)=e\mu(\varphi)$. Take
\[
\xi=\left(\inm \varphi^e\right)\bar{u}^{-1}\in\dm.
\]
Then, it is well-known (see \cite[Theorem 4.6]{KP}) that $\xi$ is transcendental over $K\!v$ and
\[
\dm=\kappa[\xi],\qquad \ka(\xi)\simeq \km. 
\]
In particular, the multiplicative group $\dm^*$ of units  of $\dm$ coincides with $\ka^*$.

The choice of the pair $\varphi, u$ determines a \textbf{residual polynomial operator} 
$$
R=R_{\mu,\varphi,u}\colon\;\kx\lra \kappa[y].
$$
Let us recall its definition. 
We agree that $R(0)=0$.
For a nonzero $g \in K[x]$ with $\varphi$-expansion $g=\sum\nolimits_{0\le n}a_n\varphi^n$, let us denote
$$
S=\left\{0\le n\mid \mu\left(a_n \varphi^n\right)=\mu(g)\right\},\quad 
\ell_0=\min(S),\quad \ell=\max(S). 
$$
If we denote $\ga=\mu(\varphi)$, then for all  $n\in\N$ we have
$$
n\in S\ \sii\  \mu(a_n)+n\ga=\mu(a_{\ell_0})+\ell_0\ga\ \sii\  (n-\ell_0)\ga=\mu(a_{\ell_0})-\mu(a_n).
$$
This implies that $(n-\ell_0)\ga$ belongs to $\gm^0$, so that  $n-\ell_0=je$ for some $j\in\N$. Since $\ell\in S$, this shows in particular that $\ell-\ell_0=de$ for some $d\in\N$. Let us denote
$$
\ell_j=\ell_0+je,\qquad  0\le j\le d.
$$
Note that $\ell_d=\ell$. Finally,  for all $0\le j\le d$, consider the \emph{residual coefficient}
\begin{equation}\label{resCoeff}
	\zeta_j=\begin{cases}
		\bar{a}_{\ell_j}/\left(\bar{a}_\ell\,\bar{u}^{d-j}\right)\in \dm^*=\kappa^*,&\quad \mbox{ if }\ \ell_j\in S,\\
		0,&\quad \mbox{ otherwise}.
	\end{cases}
\end{equation}

\defn
$\,R(g):=\zeta_0+\zeta_1\,y+\cdots+\zeta_{d-1}y^{d-1}+y^d\in \kappa[y]$.\e

Since $\ell_0\in S$, we have $\zeta_0\ne 0$. Let us display
the essential property of this operator.

\begin{Teo}\label{Hmug}\cite[Theorem 5.3]{KP}
	For all nonzero $g\in\kx$, we have 
\begin{equation}\label{Rg}
\inm g=\bar{a}_\ell\, \bar{u}^d\,\left(\inm \varphi\right)^{\ell_0}\,R(g)(\xi). 
\end{equation}
\end{Teo}

Indeed, if we denote $\ep=\bar{a}_\ell$, $\pi=\inm \varphi$, then this follows immediately from:

$$
\ep^{-1}\inm g=\sum_{\ell_j\in S}\ep^{-1}\bar{a}_{\ell_j}\,\pi^{\ell_j}=\pi^{\ell_0}\sum_{\ell_j\in S}\ep^{-1}\bar{a}_{\ell_j}\,\pi^{je}=\bar{u}^d\pi^{\ell_0}\sum_{\ell_j\in S}\zeta_j\,(\pi^e/\bar{u})^j.
$$


As a consequence, we get the following description of the set $\kpm$ \cite[Propositions 6.3 + 6.6]{KP}.

\begin{Teo}\label{charKP}
	For a residue-transcendental  $\mu$, take $\varphi\in\kpm$ of minimal degree $m$. 
	A monic $Q\in\kx$ is a key polynomial for $\mu$ if and  only if either
	\begin{itemize}
		\item $\deg(Q)=m$ \,and\; $\inm Q=\inm \varphi$, or\vskip0.1cm
		\item $\deg(Q)=me\deg(R(Q))$ \,and\; $R(Q)$ is irreducible in $\kappa[y]$.
	\end{itemize}
\end{Teo}

Now, let us see how the residual polynomial operator $R$  can  be used to determine relative residual degrees.

Suppose that $\mu<\nu$ for some $\nu\in\ttt(\g)$.
There is a natural homomorphism of graded algebras $\ggm\to\ggn$, defined as follows on homogeneous elements: 
$$\inm g\longmapsto\begin{cases}\inu g,& \mbox{ if }\mu(g)=\nu(g),\\ 0,& \mbox{  if }\mu(g)<\nu(g).\end{cases}
$$
The  image of this homomorphism is contained in the subalgebra $\ggn^0\subset\ggn$ generated by all homogeneous units \cite[Corollary 2.6]{MLV}. Restricted to the homogeneous part of grade zero, this homomorphism induces ring homomorphisms
\[
\dm\lra \Delta_\nu,\qquad \ka(\mu)\to\ka(\nu).
\]

\begin{Cor}\label{root}
	Suppose that $\mu(\varphi)=\nu(\varphi)$ and $\mu(g)<\nu(g)$. Then, the element $z:=\left(\inu\varphi^e\right)\left(\inu u\right)^{-1}\in \ka(\nu)$ is a root of $R_{\mu,\varphi,u}(g)\in\ka(\mu)[y]$. 
\end{Cor}

\begin{proof}
This follows immediately from Theorem \ref{Hmug}. The image of $\inm g$ in $\ggn$ is zero and the image of $\xi=\left(\inm\varphi^e\right)\left(\inm u\right)^{-1}$ is  equal to $z$. Since $z$ is a nonzero homogeneous element of grade zero and belongs to the image of $\ggm\to\ggn$, it must be a unit in $\Delta_\nu$; hence, it belongs to $\ka(\nu)$.  

Finally, the image of the homogeneous elements $\inm\varphi$, $\bar{a}_\ell$ and $\bar{u}$ are homogeneous units in $\ggn$ too. Hence, by taking the image in $\ggn$ of all elements involved in the identity (\ref{Rg}), we deduce that $R(g)(z)=0$.
\end{proof}\e

Now, consider the special case of an ordinary augmentation $\mu\to\nu$. That is,
\[
\nu=[\mu;\,\phi,\ga], \quad \mbox{ for some }\ \phi\in\kpm,\  \ga\in\gi, \ \ga>\mu(\phi). 
\]
By definition, $\nu$ acts on $\phi$-expansions as follows:
\[
g=\sum\nolimits_{0\le n}a_n \phi^n,\ \ \deg(a_n)<\deg(\phi)\ \imp\ \nu(g)=\min_{0\le n}\left\{\mu(a_n)+n\ga\right\}.
\]

\begin{Cor}\label{relDeg}
Suppose that $\mu\to\nu$ is an ordinary augmentation. If $\mu(\varphi)=\nu(\varphi)$, then $\ka(\nu)/\ka(\mu)$ is a simple extension, generated by $z$.  In particular, 
\[
[\ka(\nu)\colon \ka(\mu)]=\deg\left(R_{\mu,\varphi,u}(\phi)\right)=\dfrac{\deg(\phi)}{e\deg(\varphi)}.
\] 
\end{Cor}

\begin{proof}
The polynomial $\phi$ becomes a key polynomial of minimal degree for $\nu$ \cite[Corollary 7.3]{KP}. Thus, a homogeneous unit in $\ggn$ is equal to $\inu a$ for some $a\in\kx$ such that $\deg(a)<\deg(\phi)$. By the definition of the augmentation, $\inu a$ is the image of $\inm a$ under the canonical homomorphism $\ggm\to\ggn$. Hence, in this case, the image of $\ggm$ is the whole subalgebra $\ggn^0$.

In particular, restricted to homogeneous elements of grade zero, the homomorphism $\ggm\to\ggn$ induces a ring homomorphism $\dm\to\Delta_\nu$ such that 
\[
\im\left(\dm\lra \Delta_\nu\right)=\ka(\nu).
\]
Since $\dm=\ka(\mu)[\xi]$, we see that $\ka(\nu)/\ka(\mu)$ is a simple extension, generated by the image of $\xi$ in $\ggn$, which is $z$. Now, since $\phi\in\kpm$, Theorem \ref{charKP} shows that $R_{\mu,\varphi,u}(\phi)\in\ka(\mu)[y]$ is an irreducible polynomial of degree $\deg(\phi)/e\deg(\varphi)$. Thus, the computation of $[\ka(\nu)\colon \ka(\mu)]$ follows from Corollary \ref{root}.
\end{proof}\e


\section{Depth of defectless extensions}

\subsection{Generalization of Ore's result}\label{subsecD<=2}

\begin{Teo}\label{NewOre}
For a  Henselian valued field $(K,v)$, let $L/K$ be a finite simple field extension such that $(L/K,v)$ is defectless, $vL/vK$ is a cyclic group  and
 $Lv/Kv$ is a separable extension.
Then, $\dep(L/K,v)\le 2$.
\end{Teo}

\begin{proof}
	Obviously, we may assume that $[L\colon K]=e(L/K)f(L/K)>1$. Let $\oo_v\sub K$ be the valuation ring of $v$ over $K$.
	
	Let $\bar{\phi}\in K\!v[x]$ be monic, irreducible, such that $Lv=K\!v\left(\bar{z}\right)$ for some root $\bar{z}$ of $\bar{\phi}$.   Let $\phi\in\oo_v[x]\sub\kx$ be any monic lifting of $\bar{\phi}$. Clearly, \[\deg(\phi)=\deg(\bar{\phi})=f(L/K).\]
	By the Henselian property, there exists some root $z\in L$ of $\phi$, such that $v(z)=0$ and the image of $z$ in $Lv$ is $\bar{z}$. 
	
	Take $\al\in L$ such that $v(\al)>0$ and $v(\al)+vK$ generates the quotient group $vL/vK$. In other words, $e:=e(L/K)$ is the least positive integer such that $ev(\al)\in vK$.

	We shall show that $\t:=z+\al$ is a generator of $L/K$ such that $\dep(\t)\le2$.
	
Consider Taylor's expansion:
\[
\phi(\t)=\phi(z+\al)=\sum_{n\ge0}\partial_n\phi(z)\al^n,
\]
where $\partial_n$ are the linear operators of  Hasse-Schmidt. Since $\phi(z)=0$, $v(\phi'(z))=0$ (by the separability of $Lv/Kv$) and $v\left(\partial_n\phi(z)\right)\ge0$ for all $n>1$, we deduce that $v(\phi(\t))=v(\al)$.

Consider $\mu_0=[v;\,x,0]$, which is Gauss' extension of $v$ to $\kx$. Since $v(\t)=0$, we have  $\mu_0< v_\t$. Since $\bar{\phi}$ is irreducible, Theorem \ref{charKP} shows that $\phi\in\kp(\mu_0)$. Hence, we may consider the augmentation 
\[
	v\;\stackrel{x,0}\lra\;\mu_0\;\stackrel{\phi,\ga}\lra\;\mu_1,
\]
where $\ga=v(\al)=v(\phi(\t))$. Obviously, $\mu_1\le v_\t$, and this implies $\mu_1<v_\t$, because $\mu_1$ has trivial support. 
Now, $\phi$ becomes a key polynomial for $\mu_1$ of minimal degree
\cite[Corollary 7.3]{KP}. By the definition of the augmentation,  we have  $\mu_1(a)=\mu_0(a)\in \g_{\mu_0}=vK$ for all $a\in\kx$ such that $\deg(a)<\deg(\phi)$. Hence,
\[
\g_{\mu_1}^0=\{\mu_1(a)\mid a\in\kx, \ 0\le\deg(a)<\deg(\phi)\}=vK,
\]
and $e(\mu_1)=e$ because it is the least  positive integer such that $e\mu_1(\phi)=e\ga\in vK$. 

Let $g\in\kx$ be the minimal polynomial of $\t$ over $K$.
Consider the $\phi$-expansion:
\[
g=a_0+a_1\phi+\cdots +a_\ell\phi^\ell,\quad a_i\in\kx,\ \deg(a_i)<\deg(\phi), \ 0\le i\le \ell.
\]
From
\[
0=g(\t)=a_0(\t)+a_1(\t)\phi(\t)+\cdots +a_\ell(\t)\phi(\t)^\ell
\]
we deduce that $\ell=e$ and $a_\ell=1$. Indeed, for all $i$, we have
\[
\mu_0(a_i)=\mu_1(a_i)\ \imp\ \mu_0(a_i)=\mu_1(a_i)=v_\t(a_i)=v(a_i(\t)).
\]
Thus, $v(a_i(\t))\in vK$, so that 
\begin{equation}\label{values}
v(a_i(\t)\phi(\t)^i)= v(a_i(\t))+i\ga\in i\ga+vK.
\end{equation}
 
If $\ell<e$, then the $v$-value of all terms $a_i(\t)\phi(\t)^i$ would belong to differents classes of $vL/vK$. This contradicts the fact that the sum of all these terms is zero. Once we know that $\ell\ge e$, we deduce that $\ell=e$ and $a_\ell=1$, because $[L\colon K]=e\deg(\phi)$. 

This shows that $\t$ is a generator of $L/K$. Finally, let us check that $g$ is a key polynomial for $\mu_1$. 

The above argument shows that the only possible coincidence 
of two of the values displayed in (\ref{values}) is: \,$v(a_0(\t))=v(\phi(\t)^e)=e\ga$.
 
Take any  $u\in\kx$ such that $\deg(u)<\deg(\phi)$ and 
$v(u(\t))=\mu_1(u)=e\ga$. 
By definition, we have
\[
R_{\mu_1,\phi,u}(g)=\zeta_0+y,\qquad \zeta_0=\bar{a}_0\bar{u}^{-1}.
\]
Since this polynomial is irreducible, Theorem \ref{charKP} shows that $g\in\kp(\mu_1)$. Hence, we can consider the augmentation $\mu_2=[\mu_1;\,g,\infty]\le v_\t$. Since $\mu_2$ has nontrivial support, it is a leaf of $\ttt(\g)$ by Theorem \ref{existence}. Hence, $\mu_2=v_\t$ and we get a chain of ordinary augmentations:
\[
v\;\stackrel{x,0}\lra\;\mu_0\;\stackrel{\phi,\ga}\lra\;\mu_1\;\stackrel{g,\infty}\lra\;v_\t.
\]
If $\deg(\phi)>1$, this is a MLV chain, so that $\dep(\t)=2$.
If $\deg(\phi)=1$, then $\mu_1$ is a depth-zero valuation and the MLV chain of $v_\t$ would be:
\[
v\;\stackrel{\phi,\ga}\lra\;\mu_1\;\stackrel{g,\infty}\lra\;v_\t.
\]
Thus, in this case, $\dep(\t)=1$.
\end{proof}\e


The most general statement in this vein would be the following one. 

\begin{Con}\label{OreGen}
 For a  Henselian valued field $(K,v)$, let $L/K$ be a finite simple field extension such that $(L/K,v)$ is defectless. Let  $r,s$ be the minimal number of generators of $Lv/Kv$,  $vL/vK$, respectively.
Then, 
\[
\max\{r,s\}\le \dep(L/K,v)\le r+s.
\]
\end{Con}

According to this conjecture, in Theorem \ref{NewOre}, we should be able to replace the condition ``$Lv/K\!v$ separable"  with the weaker condition ``$Lv/K\!v$ simple". The following example shows that this latter condition could not be dropped.\e

\subsection{Example of an extension of depth three}
Let $k$ be an algebraically closed field of characteristic $p>0$. Take \[K=k(q,r,s)(\!(t)\!),\] where $q,r,s,t$ are indeterminates. For the valuation $v=\ord_t$,  
the valued field $(K,v)$ is Henselian, with $vK=\Z$  \,and\, $K\!v=k(q,r,s)$. 

Let $L=K(\t)$ for some root $\t\in\kb$ of the polynomial 
\[
g=\left(\left(x^p-q\right)^p-t^pr\right)^p+t^{p^4}x-t^{p^3}s\in\kx.  
\]
The term $t^{p^4}x$ is used only to ensure that $g$ is separable. If we omit it, we get a similar example with $L/K$ purely inseparable.

It can be easily shown that $\dep(\t)=3$. A MLV chain for $v_\t$ is:
\[
v\;\stackrel{x,0}\lra\;\mu_0\;\stackrel{\phi,1}\lra\;\mu_1\;\stackrel{\varphi,p^2}\lra\;\mu_2\;\stackrel{g,\infty}\lra\;v_\t,
\]
where $\phi=x^p-q$, and $\varphi=\left(x^p-q\right)^p-t^pr$.

The residual polynomials of $g$ with respect to each intermediate valuation are:
\[
R_{\mu_0,x,1}(g)=\left(y^p-q\right)^{p^2},\qquad 
R_{\mu_1,\phi,t}(g)=\left(y^p-r\right)^p,\qquad 
R_{\mu_2,\varphi,t^{p^2}}(g)=y^p-s. 
\]
Since $y^p-s$ is irreducible in $K\!v[y]$, the last computation shows that $g\in\kp(\mu_2)$. In particular, $g$ is irreducible and $[L\colon K]=p^3$. On the other hand, Corollary \ref{root}
shows that $R_{\mu_0,x,1}(g)$, $R_{\mu_1,\phi,t}(g)$ and $ 
R_{\mu_2,\varphi,t^{p^2}}(g)$ have roots in $Lv$. Hence, $q^{1/p}$, $r^{1/p}$, $s^{1/p}$ belong to $Lv$ and $[Lv\colon K\!v]=p^3$. In particular, $e(L/K)=d(L/K)=1$.

Let us show that $\dep(L/K,v)=3$. For any other generator $\al$ of $L/K$ and any MLV chain of $v_\al$, Corollary \ref{relDeg} shows that in each augmentation $\mu_i\to\mu_{i+1}$ the extension $\ka(\mu_i)\to\ka(\mu_{i+1})$ is simple. Hence, the extension $Lv/K\!v$ is obtained as a tower of simple extensions. Since
\[
 Lv=K\!v\left(q^{1/p},r^{1/p},s^{1/p}\right),
 \]
 we need at least three simple extensions to construct $Lv$ from $K\!v$.

Similar examples can be constructed, with $\dep(L/K,v)$ arbitrarily large.

\subsection{Algebraic elements of depth one}\label{subsecDep1}

\begin{Teo}\label{dep1}
Let $(L/K,v)$ be a  finite, simple extension, which is unibranched and defectless. Then, the following conditions are equivalent.
\begin{enumerate}
\item [(a)] $\dep(L/K,v)=1$.
\item [(b)] There exists  $\al\in L$ such that $v(\al)+vK$ generates $vL/vK$ and $\left(\al^e/u\right)\!v$ generates $Lv/K\!v$, where $e=e(L/K)$ and $u\in K$ has $v(u)=ev(\al)$.  
\end{enumerate}
\end{Teo}

\begin{proof}
	Suppose that $\dep(L/K,v)=1$. Let $\t\in L$ be a generator such that the valuation $\vt$ admits an MLV chain of length one:
\[
v\;\stackrel{x-a,\ga}\lra\;\mu_0\;\stackrel{g,\infty}\lra\;\vt,
\]
where $g$ is the minimal polynomial of $\t$ over $K$.
We claim that the augmentation $\mu_0\to \vt$ is necesssarily ordinary.
Indeed, suppose that it is a limit augmentation. By \cite[Lemma 5.3]{MLV}, we would have
\[
K\!v=\ka(\mu_0)=\ka(\vt)=Lv,
\]
so that $f(L/K)=1$. On the other hand, $\max\left(D_1(\t)\right)$ would not exist, by Theorem \ref{OSdepth}. This implies $e(L/K)=1$ \cite{B}. Since $(L/K,v)$ is unibranched, $g$ is irreducible over $\khx$, so that 
\[
\deg(g)=e(L/K)f(L/K)d(L/K)=1,
\]
because we are assuming that $d(L/K)=1$. Hence, $\t\in K$ and $\vt$ is a depth-zero valuation. This contradicts our assumption.

Therefore, $g\in\kp(\mu_0)$ and $\deg(g)>1$. Obviously,
\[
vL=\g_{\vt}=\g_{\mu_0}=\gen{vK,\ga}_\Z.
\]
Hence, $e:=e(L/K)$ is the least positive integer such that $e\ga\in vK$. Choose any $u\in K$ such that $v(u)=ev(\ga)$. 
 
The element $\al:=\t-a\in L$ satisfies the conditions of (b). Since $v(\al)=\vt(x-a)=\mu_0(x-a)=\ga$, we see that $v(\al)+vK$ generates the quotient group $vL/vK$. On the other hand, $\inn_{\vt}(x-a)^e \bar{u}^{-1}$ generates $\ka(\vt)/K\!v$  by Corollary \ref{relDeg}. By applying the natural isomorphism described in (\ref{natIso}), determined by $x\mapsto \t$, we see that $\left(\al^e/u\right)\!v$ generates $Lv/K\!v$.
This proves that (a) implies (b).

Let us show that (b) implies (a). Take $\al\in L$ satisfying the conditions of (b). Let $\ga=v(\al)$. By our assumption, $e:=e(L/K)$ is the least positive integer such that $ev(\ga)$ belongs to $vK$.  Consider the depth-zero valuation
\[
 v\;\stackrel{x,\ga}\lra\;\mu_0.
\]
Since $v(\al)=\ga$, we have $\mu_0<v_\al$.

Take any $u\in K$ such that $v(u)=ev(\ga)$. Let $g\in\kx$ be the minimal polynomial of $\al$ over $K$. Since
\[
\mu_0(x)=v(\al)=v_\al(x)\quad\mbox{ and }\quad  \mu_0(g)<v_\al(g)=\infty,
\]
Corollary \ref{root}  shows that $R_{\mu_0,x,u}(g)$ vanishes on 
$\left(\inn_{v_\al}x^e\right)/\bar{u}$. Through the isomorphism $\ka(v_\al)\simeq Lv$ given in (\ref{natIso}), this element corresponds to  
$\left(\al^e/u\right)\!v$. By our assumption, this element has degree $f(L/K)$ over $K\!v$. Hence, 
\[
\deg g\ge e(L/K) \deg\left(R_{\mu_0,x,u}(g)\right)\ge e(L/K)f(L/K)=[L\colon K].
\]
This proves that $\deg(g)=[L\colon K]$ and  $\deg\left(R_{\mu_0,x,u}(g)\right)=f(L/K)$. In particular,  
$\al$ generates $L/K$ and $R_{\mu_0,x,u}(g)$ is irreducible. By Theorem \ref{charKP}, $g$ is a key polynomial for $\mu_0$. Hence, it makes sense to consider the ordinary augmentation
\[
\mu_0 \;\stackrel{g,\infty}\lra\;\mu_1\le v_\al.
\]
Since $\mu_1$ has nontrivial support, it is a leaf of $\ttt(\g)$ by Theorem \ref{existence}. Hence, necessarily $\mu_1=v_\al$, and we get a MLV chain of $v_\al$: 
\[
 v\;\stackrel{x,\ga}\lra\;\mu_0\;\stackrel{g,\infty}\lra\;v_\al.
\]
This shows that $\dep(\al)=1$, so that $\dep(L/K,v)=1$.
\end{proof}\e

It is particularly easy to check condition (b) of Theorem \ref{dep1} in the extreme cases $[L\colon K]=e(L/K)$ or 
$[L\colon K]=f(L/K)$. 

\begin{Cor}
A finite simple extension $(L/K,v)$ has depth one in any of the following two situations.
\begin{enumerate}
 \item[(i)] $L/K$ is totally ramified and $vL/vK$ is a cyclic group.
 \item[(ii)] $[L\colon K]=f(L/K)$ and $Lv/K\!v$ is simple.
\end{enumerate}
\end{Cor}

On the other hand, condition (b) implies that $\op{gr}_v(L)=\op{gr}_v(K)[\inv \al]$. This monogeneity condition is not sufficient to guarantee that   $\dep(L/K,v)=1$, as illustrated by the example in Section \ref{subsecD2}.

\subsection{A tame extension of depth two}\label{subsecD2}
Take the 2-adic field,  $K=\Q_2$, equipped with the 2-adic valuation $v=\ord_2$. Consider the extension $L=K(\al,\ep)$, where
\[
\al^3=2, \qquad \ep^3=1,\ \ep\ne 1.  
\]
The minimal polynomial of $\ep$ over $K$ is $x^2+x+1$, which is irreducible modulo 2. Hence, $K(\ep)/K$ is an unramified subextension of $L/K$. Since $v(\al)=1/3$, we have $[L\colon K(\ep)]=3$. Thus,
\[
[L\colon K]=6,\qquad f(L/K)=2,\qquad e(L/K)=3.
\]
Hence $L/K$ is \textbf{tame}. This means, by definition, that:
\[
Lv/K\!v \ \mbox{ separable},\qquad d(L/K)=1,\qquad 2\nmid e(L/K).    
\]
Let us show that $\dep(L/K,v)=2$, by checking that no generator of $L/K$ satisfies the condition (b) of Theorem \ref{dep1}.

Suppose that $\t$ is a generator of $L/K$ such that $v(\t)+vK$ generates $vL/vK$. We can express every $\t\in L$ in a unique way as:
\[
 \t=a+b\al+c\al^2,\quad a,b,c\in K(\ep).  
\]
The three values $v(a)$, $v(b\al)$, $v(c\al^2)$ are different, because they have different images in the quotient $vL/vK\simeq  \Z/3\Z$. Hence, 
\[
 v(\t)=\min\{v(a),\,v(b\al),\,v(c\al^2)\}.
\]
Since the class $v(\t)+vK$ generates $vL/vK$, either $v(\t)=v(b\al)$, or $v(\t)=v(c\al^2)$. Let us first assume that 
$v(\t)=v(b\al)$. Then,
\[
v(\t-b\al)=\min\{v(a),\,v(c\al^2)\}>v(\t).
\]
Let $b=2^mB$, for some $B\in K(\ep)$ such that $v(B)=0$. Then, $v(\t^3)=1+3m$ and we may take $u=2^{3m+1}$ as an element in  $K$  with the same value. Now,
\[
\left(\t^3/u\right)\!v=\left((b\al)^3/u\right)\!v=\left(B^3\right)\!v=\left(Bv\right)^3=1, 
\]
because $K(\ep)v$ is the finite field with 4 elements. Thus, $\t$ does not satisfy the condition  (b) of Theorem \ref{dep1}.

The discussion of the case $v(\t)=v(c\al^2)$ is completely analogous. 
This ends the proof that  $\dep(L/K,v)=2$. 

Nevertheless, the graded algebra $\op{gr}_v(L)$ is generated by a single element as a $\op{gr}_v(K)$-algebra. Indeed, take $\t=\al(\ep-1)$. Since $v(\t)=1/3$, the algebra $\op{gr}_v(K)[\inv \t]$ has the same group of grades as $\op{gr}_v(L)$. In order to check that $\op{gr}_v(K)[\inv \t]=\op{gr}_v(L)$  it suffices to check that the subalgebra contains the subfield $K(\ep)v$.  Now, 
\[
\t^3=2(\ep-1)^3=6(1+2\ep)\ \imp\ \left(\t^3-2\right)/4=1+3\ep. 
\]
Hence, $\left((\t^3-2)/4\right)\!v$ generates the field $K(\ep)v$ over $K\!v$.


\section{An example in the  non-Henselian case}

In this section, we study a concrete branched extension $(L/K,v)$  inspired in \cite[Example 8.6]{NN}. The generating polynomial $g\in\kx$ has  degree  four and splits completely in $\khx$. Thus, for each root $\t\in K^h$ of $g$, the valuation $v_\t$ has depth zero as a valuation on $\khx$. However, we shall see that its restriction to $\kx$ acquires depth two.
This suggests that the number $\dep(\t)$ has a relevant arithmetical meaning only for unibranched extensions.    

Consider the field $K=\Q(t)$ equipped with the $\ord_t$ valuation. Every $u\in K^*$ has an initial term 
$$
\op{in}(u):=\left(u\,t^{-\ord_t(u)}\right)(0)\in \Q^*.
$$

Take a prime number $p\equiv1\md4$ and let $\ord_p$ be the $p$-adic valuation.

Consider the following discrete rank-two valuation on $K$:
$$
v\colon K^*\lra \Z^2_{\op{lex}},\qquad v(u)=\left(\ord_t(u),\ord_p(\op{in}(u))\right).
$$ 



Let  $i\in\kb$ be a root  of the polynomial $x^2+1$. Consider its $p$-adic expansion
$$
i=i_0+i_1p^{\ell_1}+\cdots+ i_np^{\ell_n}+\cdots,
$$
with $0<i_n<p$ for all $n\ge0$. Denote the truncations of $i$ by
$$
a_n=i_0+i_1p^{\ell_1}+\cdots+ i_{n-1}p^{\ell_{n-1}}\in\Z.
$$

On the other hand, consider $\al\in\kb$ defined as
$$
\al=\sqrt{1+t}=1+(1/2)t+j_2t^2+\cdots +j_nt^n+\cdots,
$$
where $j_n\in\Q$ for all $n\ge0$. Denote the truncations of $\al$ by
$$
b_n=j_0+j_1t+\cdots+ j_{n-1}t^{n-1}\in K,
$$
where $j_0=1$ and $j_1=1/2$.

\begin{Obs}\label{rem1}
	Both, $i$ and $\al$ belong to $K^h$.	
\end{Obs}

Indeed, suppose $\be\in\kb$ satisfies $\be^2\in K$ and $v(\be-a)>v(a)$ for some $a\in K$. Then, for every $\sg$ in the decomposition group $\dd
=\op{Aut}(\kb/K^h)$ we must have $\sg(\be)=\be$, because the equality $\sg(\be)=-\be$ leads to a contradiction. Since $v\circ\sg=v$, we have:
\[
v(\be-a)=v(\sg(\be)-\sg(a))=v(-\be-a)=v(\be+a)=v(\be-a+2a)=v(a).
\]

Consider the extension 
\[
L=K(\t),\qquad \mbox{where}\quad \t=\al+i\in K^h.
\] 
The minimal polynomial of $\t$ over $K$ is: 
\[
g=x^4-2t\,x^2+(t+2)^2\in \kx. 
\]
We have $\op{Z}(g)=\left\{\pm \t, \pm\t'\right\}\sub K^h$, where $\t':=\al-i=(t+2)/\t$. This follows from:
\[
t+2=\al^2+1=(\al+i)(\al-i).
\]
Therefore, $L/K$ is a Galois extension. Also, since $\al+i$ and $\al-i$ both belong to $L$, we see that $i$ and $\al$ both belong to $L$. Thus, the Galois group of $L/K$ is $C_2\times C_2$ and the network of subextensions is 

\begin{center}
	\setlength{\unitlength}{4mm}
	\begin{picture}(12,9.4)
		\put(-0.2,4){\begin{footnotesize}$K(i)$\end{footnotesize}}\put(3.3,4){\begin{footnotesize}$K(i\al)$\end{footnotesize}}\put(7.2,4){\begin{footnotesize}$K(\al)$\end{footnotesize}}
		\put(4,-0.1){\begin{footnotesize}$K$\end{footnotesize}}\put(4,8){\begin{footnotesize}$L$\end{footnotesize}}
		\put(1,5){\line(1,1){2.7}}\put(4.3,5){\line(0,1){2.7}}\put(7.8,5){\line(-1,1){2.7}}
		\put(1,3.5){\line(1,-1){2.7}}\put(4.4,3.5){\line(0,-1){2.7}}\put(7.8,3.5){\line(-1,-1){2.7}}
	\end{picture}
\end{center}\bs

By Remark \ref{rem1}, $L\sub K^h$. Hence, \cite[Corollary 3.2]{NN} shows that $v$ has four immediate extensions to the quartic field $L$. The valuations on $\kx$ asociated to these extensions with respect to the choice of $g$ as a generator of $L/K$ are:
\[
\vt,\quad v_{-\t},\quad v_{\t'},\quad v_{-\t'}.
\]  
The four extensions of $v$ to $L$ are defectless. If $w$ denotes any one of this extensions, we have necessarily
\[
e(w/v)=f(w/v)=d(w/v)=1.
\]

\subsection*{Computation of the depth of $\t$}
Let us compute an Okutsu sequence of $\t$.

The distance set $D_1=D_1(\t,K)$ has no maximal element and, for
the family $A_0:=\{a_n+1\mid n\in\N\}$, the values 
\[
v(\t-(a_n+1))=(0,\ell_n), \quad \mbox{ for all }n\in\N,
\]
are cofinal in $D_1$. 

Also, $v(\t-i-1)=(1,0)>D_1$, the set $D_2$ has no maximal value and, for the set $A_1:=\{i+b_m\mid m\in\N\}$, the values 
\[
v(\t-(i+b_m))=(m,\ord_p(j_m))\in D_2
\]
are unbounded.
Therefore, $\left[A_0,A_1,A_2=\{\t\}\right]$ is an Okutsu sequence of $\t$. By Theorem \ref{ballsDepth}, the valuation $\vt$ has depth two and the MLV chains of $\vt$ consist of two consecutive limit augmentations.

Obviously, by applying the automorphisms in $\op{Gal}(L/K)$ to the Okutsu sequence of $\t$, we obtain Okutsu sequences of the same length of the other three roots of $g$. Thus, the four valuations $\vt$, $v_{-\t}$, $v_{\t'}$, $v_{-\t'}$ have depth two.  

\subsection*{Computation of the depth of all generators of $L/K$}
Let us show that all generators of $L/K$ have depth two too.

Clearly, $L=\langle1,\,i,\,\al,\,i\al\rangle_K$. Hence, every $\eta \in L$ is of the form
\[
\eta=a+b\, i+c\,\al+d\,i\al,\quad\mbox{ for some}\quad a,b,c,d\in K.
\]

For any choice of $r,s\in K$, with $r\ne0$, the length of the Okutsu sequences of $\eta$ and $r\eta+s$ coincide. Thus, we can assume that 
\[
\eta=b\, i+c\,\al+d\,i\al,\quad\mbox{ and}\quad \min\{v(b),v(c),v(d)\}=(0,0).
\]

As a consequence, as an element in $\Q_p(\!(t)\!)$, $\eta$ can be written in a unique way as
\[
\eta=\sum_{m\ge0}i\,q_mt^m + \sum_{n\ge0}r_nt^n, 
\]
for some $q_m,\,r_n\in\Q\cap\Z_p$. 

Suppose that $L=K(\eta)$. Then, since $\eta\not\in K(\al)$, there exists a minimal $m_0\in\N_0$ such that $q_{m_0}\ne0$. Now, consider the set $A_0:=\{c_n\mid n\in\N\}\sub K$, where
\[
c_n=a_n\,q_{m_0}t^{m_0} + \sum_{k=0}^{m_0}r_kt^k, \quad n\in\N.
\]

Clearly, the values $v(\eta-c_n)=(m_0,\ell_n+\ord_p(q_{m_0}))$ are cofinal in the set $D_1(\eta,K)$, and all these values are upper bounded by 
$(m_0+1,0)$.

Finally, consider the set $A_1:=\{\be_N\mid N\in\N\}\sub K(i)$, where
\[
\be_N:=\sum_{0\le m<N}i\,q_mt^m + \sum_{0\le n<N}r_nt^n.
\]
Clearly, the values $v(\eta-\be_N)>(N-1,0)$ are unbounded, so that
$\left[A_0,A_1,\{\eta\}\right]$ is an Okutsu sequence of $\eta$.
By Theorem \ref{OSdepth}, $\dep\left(L/K,v\right)=2$ and the MLV chains of $v_\eta$ consist of two consecutive limit augmentations.


\end{document}